\documentclass{article}

\usepackage{makeidx}
\usepackage{latexsym}
\usepackage{amsfonts}
\usepackage{amssymb}
\usepackage{amsmath}
\usepackage{amstext}
\usepackage{amsthm}
\usepackage{mathrsfs}

\newtheorem*{thmA}{Theorem A}
\newtheorem*{thmB}{Theorem B}
\newtheorem*{thmC}{Theorem C}
\newtheorem*{thmD}{Theorem D}
\newtheorem{theorem}{Theorem}

\newtheorem{lemma}[theorem]{Lemma}

\newtheorem{corollary}[theorem]{Corollary}

\newtheorem{proposition}[theorem]{Proposition}

\begin{document}

\title{On the reversal bias of the\\ Minimax
social choice correspondence\footnote{Daniela Bubboloni was partially supported by GNSAGA
of INdAM.}}
\author{\textbf{Daniela Bubboloni} \\
{\small {Dipartimento di Scienze per l'Economia e  l'Impresa} }\\
\vspace{-6mm}\\
{\small {Universit\`{a} degli Studi di Firenze} }\\
\vspace{-6mm}\\
{\small {via delle Pandette 9, 50127, Firenze, Italy}}\\
\vspace{-6mm}\\
{\small {e-mail: daniela.bubboloni@unifi.it}}\\
\vspace{-6mm}\\
{\small tel: +39 055 2759667} \and \textbf{Michele Gori}
 \\
{\small {Dipartimento di Scienze per l'Economia e  l'Impresa} }\\
\vspace{-6mm}\\
{\small {Universit\`{a} degli Studi di Firenze} }\\
\vspace{-6mm}\\
{\small {via delle Pandette 9, 50127, Firenze, Italy}}\\
\vspace{-6mm}\\
{\small {e-mail: michele.gori@unifi.it}}\\
\vspace{-6mm}\\
{\small tel: +39 055 2759707}}

\maketitle

\begin{abstract}
We introduce three different qualifications of the reversal bias
in the framework of social choice correspondences.
For each of them, we prove that the Minimax social choice correspondence is immune to it if and only if
the number of voters and the number of alternatives satisfy suitable arithmetical conditions.
We prove those facts thanks to a new characterization of the Minimax social choice correspondence and using a graph theory approach. We discuss the same issue for the Borda and Copeland social choice correspondences.
\end{abstract}

\vspace{4mm}

\noindent \textbf{Keywords:} reversal bias; Minimax social choice correspondence; directed graphs.

\vspace{2mm}

\noindent \textbf{JEL classification:} D71.

\noindent \textbf{MSC classification:} 05C20.

\section{Introduction}

Consider a committee having $h\ge 2$ members who have to select one or more elements within a set of $n\ge 2$ alternatives. Usually, the procedure used to make that choice only depends on committee members' preferences on alternatives. We assume that preferences of committee members are expressed as strict rankings (linear orders) on the set of alternatives, and call preference profile any list of $h$ preferences, each of them associated with one of the individuals in the committee. Thus, a procedure to choose, whatever individual preferences are, one or more alternatives as social outcome can be represented by a social choice correspondence ({\sc scc}), that is, a function from the set of preference profiles to the set of nonempty subsets of the set of alternatives.

The assessment of different {\sc scc}s and their comparison is usually based on which properties, among the ones considered desirable or undesirable under a social choice viewpoint, those {\sc scc}s fulfil.
Moving from the ideas originally proposed by Saari (1994) and then deepened by Saari and Barney (2003), we focus here on a quite unpleasant property that a {\sc scc} may meet and that, in our opinion, hasn't deserved the right attention yet.

In order to describe such a property, recall that the reversal of a preference profile is the preference profile obtained by it assuming a complete change in each committee member's mind about her own ranking of alternatives (that is, the best alternative gets the worst, the second best alternative gets the second worst, and so on).
Assume now that a given {\sc scc} associates with a certain preference profile a singleton, that is, it selects a unique alternative. If we next consider the outcome determined by the reversal of the considered preference profile, we would expect to have something different from the previous singleton as it seems natural to demand a certain degree of difference between the outcomes associated with a preference profile and its reversal.  As suggested by Saari and Barney (2003, p.17),
\begin{quote}
{\small suppose after the winner of an important
departmental election was announced, it was discovered
that everyone misunderstood the chair's instructions.
When ranking the three candidates, everyone listed his top,
middle, and bottom-ranked candidate in the natural order
first, second, and third. For reasons only the chair understood,
he expected the voters to vote in the opposite way.
As such, when tallying the ballots, he treated a first and
last listed candidate, respectively, as the voter's last and
first choice.
Imagine the outcry if after retallying the ballots the chair
reported that [...] the same person won. }
\end{quote}
In other words,
common sense suggests that we should
express doubts about the quality of a {\sc scc} which associates the same singleton both with a preference profile and with its reversal, that is, which suffers what we are going to call the reversal bias.

Among the classical {\sc scc}s, such a bias is experienced by the Minimax {\sc scc}\footnote{ Also known as Simpson-Kramer or Condorcet {\sc scc}. },  that is the {\sc scc} which selects those alternatives whose greatest pairwise defeat is minimum. Indeed, assume that a committee having six members ($h=6$) has to select some alternatives within a set of four alternatives denoted by 1, 2, 3 and 4 ($n=4$). Consider then a  preference profile represented by the matrix
\[
\left[
\begin{array}{cccccccccc}
1 & 1 & 1 & 2 & 3 & 4 \\
2 & 3 & 4 & 3 & 4 & 2 \\
3 & 4 & 2 & 4 & 2 & 3 \\
4 & 2 & 3 & 1 & 1 & 1
\end{array}
\right]
\]
where, for every $i\in\{1,2,3,4,5,6\}$, the $i$-th column represents the $i$-th member's preferences according to the rule that the higher the alternative is, the better it is.
A simple check shows that the Minimax {\sc scc} associates both with that preference profile and with its reversal the same set $\{1\}$.
On the other hand, if we consider two alternatives only, then  the Minimax {\sc scc} agrees with the simple majority and it is immediate to verify that it is immune to the reversal bias whatever the number of committee members is.

For such a reason, we address the problem of finding conditions on the number of individuals and on the number of alternatives that make the Minimax {\sc scc} immune to the reversal bias.
Our main result\footnote{Theorem A is a rephrase of Theorem \ref{main-new} for $j=1$.} is the following theorem.
\begin{thmA}
The Minimax {\sc scc} is immune to the reversal bias if and only if  $h\le 3$ or $n\le 3$ or
	$(h,n)\in\{(4,4), (5,4), (7,4), (5,5)\}$.
\end{thmA}
Theorem A shows, in particular, that the Minimax {\sc scc} does no exhibit the reversal bias not only when there are two alternatives but also in other cases. Remarkably, that property holds true when alternatives are three, independently on the number of individuals, and when individuals are three, independently on the number of alternatives.

The proof of Theorem A requires a certain amount of work and the use of language and methods taken from graph theory\footnote{Note that the use of graphs in social choice theory is well established (see, for instance, Laslier (1997)).}.
Indeed, standard social choice theoretical arguments naturally allow to prove that, for lots of pairs $(h,n)$, the Minimax {\sc scc} suffers the reversal bias\footnote{Propositions \ref{blacktable-odd} and \ref{blacktable-even}, which determine a large set of pairs $(h,n)$ for which the Minimax {\sc scc} suffers the reversal bias,
are based on an intuitive argument which could be carried on without the machinery of graph theory.}. On the other hand, except for the trivial case $n=2$,
they turn out to be difficult to apply to prove that, for the remaining pairs, the Minimax {\sc scc} is immune to the reversal bias. In particular, no simple intuition indicates how to treat the cases $(h,n)\in\{(4,4), (5,4), (7,4), (5,5)\}$.
For such a reason, we first propose a new characterization of the Minimax {\sc scc} showing that, for every preference profile, an alternative $x$ is selected by the Minimax {\sc scc} if and only if, for every majority threshold $\mu$ not exceeding the number of individuals but exceeding half of it, if
there is an alternative which is preferred by at least $\mu$ individuals
to $x$, then, for every alternative, there is another one which is preferred by at least $\mu$ individuals
to it (Proposition \ref{Kramer}).
We then associate with each preference profile $p$ and each majority threshold $\mu$ a directed graph $\Gamma_{\mu}(p)$, called majority graph,
whose vertices are the alternatives and whose  arcs are the $\mu$-majority relations among alternatives (Section \ref{majority graphs}). By the analysis of connection and acyclicity properties of those graphs, we find out a general and unified method to approach the proof of Theorem A. That allows, in particular, to avoid the repetition of similar arguments and the discussion of very long lists of cases and subcases. The geometric representation of the graph $\Gamma_{\mu}(p)$ is also a useful mental guidance in the tricky steps needed to carry on such an analysis as well as the proof of Theorem A.
We emphasise that the results related to graph theory deal with quite general majority issues so that they are not limited, in their meaning, to the specific problem considered in the paper.
We are confident that those results could be a smart tool to manage, in the future, many other problems related not only to the Minimax {\sc scc}.

We also introduce two weaker versions of reversal bias. Namely, we say that a {\sc scc} suffers the reversal bias of type 2 if
there exists a preference profile such that the outcomes associated with it and its reversal are not disjoint and one of the two is a singleton;
we say instead that a {\sc scc} suffers the reversal bias of type 3 if there exists a preference profile such that the outcomes associated  with it and its reversal are not disjoint and none of the two is the whole set of the alternatives.
It is immediate to observe that the reversal bias (also called reversal bias of type 1) implies the reversal bias of type 2 which in turn implies  the reversal bias of type 3. Using the same tools and techniques used to prove Theorem A, we get the following results\footnote{Theorems B and C are rephrases of Theorem \ref{main-new} for $j=2$ and $j=3$, respectively.}.

\begin{thmB}
The Minimax {\sc scc} is immune to the reversal bias of type 2 if and only if  $h=2$ or $n\le 3$ or $(h,n)=(4,4)$.\vspace{-1mm}
\end{thmB}

\begin{thmC}
The Minimax {\sc scc} is immune to the reversal bias of type 3 if and only if $n=2$ or $(h,n)=(3,3)$.
\end{thmC}

We emphasize that there is an interesting link between the different qualifications of reversal bias above described and the concept of Condorcet loser. Indeed,  let $C$ be a {\sc scc} satisfying the Condorcet principle, that is, always selecting the Condorcet winner as unique outcome when it exists. If $C$ is immune to the reversal bias of type 1, then it never selects the Condorcet loser as unique outcome, that is, $C$ fulfils the weak Condorcet loser property; if $C$ is immune to the reversal bias of type 2, then it never selects the Condorcet loser, that is, $C$ fulfils the Condorcet loser property.
Thus, since the Minimax {\sc scc} satisfies the Condorcet principle, Theorems A and B provide, in particular, conditions on $(h,n)$ that are sufficient to make the Minimax {\sc scc} satisfy the weak Condorcet loser property and the Condorcet loser property, respectively. Certainly, as it is not known whether such conditions are also necessary, determining all the pairs $(h,n)$ making the Minimax {\sc scc} satisfy those properties is an interesting problem which, in our opinion, can be fruitfully attacked using the methods described in this paper.
Finally note that, given a {\sc scc} $C$ always selecting the Condorcet winner (not necessarily as unique outcome) when it exists, we have that if $C$ is immune to the reversal bias of type 2, then it fulfils the weak Condorcet loser property; if $C$ is immune to the reversal bias of type 3, then it never selects the Condorcet loser when the set of outcomes is different from the whole set of alternatives.

Observe now that, even though the main concepts of our paper are mainly inspired to the ideas of Saari and Barney (2003), the framework we consider, as well as the terminology we use, is different from the one used by those authors. Indeed, they deal with election methods, namely, functions from the set of  finite sequences of individual preferences (still called preference profiles) to the set of complete and transitive relations on the set of alternatives.
In that framework, they say that an election method suffers the reversal bias if it associates the same relation with a preference profile and its reversal, provided that such a relation is not a complete tie, so that in their paper the expression reversal bias is used with a different meaning. For every $k\le n-1$, they also introduce the concept of $k$-winner reversal bias (called top-winner bias when $k=1$), that is that phenomenon that occurs when an election method associates with a preference profile and its reversal two relations having the property to have the same $k$ top ranked alternatives\footnote{Saari (1994) introduces for election methods another interesting concept, called reversal symmetry, which is related to the ones now discussed. Namely, an election method is said to be reversal symmetric if the outcomes associated with any preference profile and its reversal are one the reversal of the other. Of course, if an election method is reversal symmetric it cannot suffer either the reversal bias or the $k$-winner reversal bias. Reversal symmetry has been recently studied
by Llamazares and Pe$\mathrm{\tilde{n}}$a (2015) for positional methods and by Bubboloni and Gori (2015) for social welfare functions with values in the set of linear orders. }. Anyway, despite the differences, it is obvious that any result of theirs about the top-winner reversal bias of a certain election method implies some information about the reversal bias of type 1 for the {\sc scc} generated by that method restricting its domain to those sequences of individual preferences having $h$ terms and looking only at those alternatives that are top ranked. On the other hand, it is clear that none of their theorems implies a result about the reversal biases of type 2 and 3 as an immediate by-product. In particular, from Theorem 8 in Saari and Barney (2003), we deduce that the Borda and Copeland {\sc scc}s are immune to the reversal bias of type 1, but nothing can be deduced about the other types of reversal bias. That makes interesting the following result\footnote{Theorem D is a rephrase of Proposition \ref{sym3}.}.
\begin{thmD}
The Borda and Copeland {\sc scc}s are immune to the reversal bias of type 3.
\end{thmD}

We conclude with an observation. Recall that a positional method is an election method where each time an alternative is ranked $k$-th by one individual it obtains $w_k$ points and alternatives are then ranked according to the final score they get; the vector $w=(w_k)_{k=1}^n\in\mathbb{R}^n$ associated with the method is called its voting vector and is assumed to satisfy the relations $w_1\ge w_2\ge \ldots\ge w_n$ and $w_1>w_n$. Assume now that $n\ge 3$ and consider a voting vector $w$ such that there exist $k_1,k_2\in\{1,\ldots,n\}$ with $w_{k_1}+w_{n-k_1}\neq w_{k_2}+w_{n-k_2}$.  Then, from Theorem 1 in Saari and Barney (2003), we deduce that the {\sc scc} generated by the positional method associated with $w$ suffers the reversal bias of type 1, provided that the number $h$ of individuals is large enough.
That fact is remarkable  because it implies the existence of many {\sc scc}s different from the Minimax {\sc scc}, like plurality and anti-plurality {\sc scc}s, which suffer that bias.
Certainly, as the considered theorem gives no information about the exact values of $h$ for which the reversal bias of type 1 really occurs, finding those values of $h$ is an interesting issue that deserves to be carefully investigated.
More generally, given any classical {\sc scc} $C$ and any social choice property, one can consider the problem to determine
conditions on the number of individuals and alternatives which are necessary and sufficient to make $C$ fulfil the property.
We believe that investigating those problems is an interesting and promising research project since, as particularly shown by our results on the reversal bias, comparing different {\sc scc}s on the basis of their properties cannot ignore how many individuals and alternatives are involved in the decision process.


\section{Preliminary definitions}

Let $\mathbb{N}_\diamond=\{a\in\mathbb{N}:a\ge 2\}$.
From now on, let $n,h\in \mathbb{N}_\diamond$ be fixed, and let $N=\{1,\ldots,n\}$ be the set of alternatives and $H=\{1,\ldots,h\}$ be the set of individuals.

A {\it preference relation} on $N$ is a linear order on $N$, that is, a complete, transitive and antisymmetric binary relation on $N$. The set of linear orders on $N$ is denoted by $\mathcal{L}(N)$. Let $q\in\mathcal{L}(N)$ be fixed. Given  $x,y\in N$, we  usually write $x\ge_{q}y$ instead of $(x,y)\in q$, and $x>_{q}y$ instead of $(x,y)\in q$ and $x\neq y$. The function $\mathrm{rank}_{q}:N\to \{1,\ldots,n\}$ defined, for every $x\in N$, by $\mathrm{rank}_q(x)=|\{y\in N: y>_q x\}|+1$ is
bijective.
We identify $q$ with the function $\mathrm{rank}_q^{-1}$ and denote it still by $q$. We also identify $q$ with the column vector $[q(1),\dots,q(n)]^T$.
Moreover, we define $q^r$ as the element in $\mathcal{L}(N)$ such that, for every $x,y\in N$,
$(x,y)\in q^r$ if and only if $(y,x)\in q$. Of course, $(q^r)^r=q$.
For instance, let $n=3$ and $q\in\mathcal{L}(N)$ be such that $2>_q 1>_q 3$. Then
  $q(1)=2$, $q(2)=1$, $q(3)=3$ and we identify $q$ with  $[2,1,3]^T$ and $q^r $ with $[3,1,2]^T.$

A {\it preference  profile} is an element of $\mathcal{L}(N)^h$. The set $\mathcal{L}(N)^h$ is denoted by $\mathcal{P}$.
Let $p\in\mathcal{P}$ be fixed. Given $i\in H$, the $i$-th component of $p$ is denoted by $p_i$ and represents the preferences of individual $i$. The preference profile  $p$ can be naturally identified with the matrix whose $i$-th column is $[p_i(1),\dots,p_i(n)]^T$.
Define $p^{r} \in \mathcal{P}$ as the preference profile such that, for every $i\in H$,
$(p^{r})_i= (p_i)^r$. Of course, $(p^r)^r=p$.
We will write the $i$-th component of  $p^{r}$ simply as $p^{r}_i,$ instead of $\left(p^{r}\right)_i$. Given $\mu\in \mathbb{N}\cap (h/2,h]$ and $x,y\in N$, we write $x>_\mu^p y$ if
$|\{i\in H: x>_{p_i} y\}|\ge \mu$. Note that $x>_\mu^p y$ if and only if  $y>_\mu^{p ^r}x.$
Elements in $ \mathbb{N}\cap (h/2,h]$ are called  majority thresholds. We call  {\it minimal majority threshold} the integer $\mu_0=\lceil \frac{h+1}{2}\rceil$. Further details about preference relations and preference profiles can be found in Bubboloni and Gori (2015).

A {\it social choice correspondence} ({\sc scc}) is a function from $\mathcal{P}$ to the set of the nonempty subsets of $N$.
The set of {\sc scc}s is denoted by $\mathfrak{C}$. Let $C\in \mathfrak{C}$. We say that $C$ suffers the {\it reversal bias (of type 1)} if, there exists $p\in\mathcal{P}$ and $x\in N$ such that
\[
 C(p)= C(p^r)=\{x\};
\]
the {\it reversal bias of type 2} if, there exists $p\in\mathcal{P}$ such that
\[
|C(p)|=1 \mbox{ and } C(p)\cap C(p^r)\neq\varnothing;
\]
the {\it reversal bias of type 3} if, there exists $p\in\mathcal{P}$ such that
\[
|C(p)|<n \mbox{ and } C(p)\cap C(p^r)\neq\varnothing.
\]
Clearly if $C\in \mathfrak{C}$ suffers the reversal bias of type 1, then $C$ suffers also the reversal bias of type 2, and if $C$ suffers the reversal bias of type 2 then $C$ suffers also the reversal bias of type 3. We define, for every $j\in\{1,2,3\}$, the sets
\[
\mathfrak{C}^{j}=\{C\in \mathfrak{C}: C\ \hbox{is immune to the reversal bias of type}\ j\}.
\]
 Note that $\mathfrak{C}^3\subseteq \mathfrak{C}^2\subseteq \mathfrak{C}^1$.

\section{The Minimax {\sc scc}}

In this section we focus on the  Minimax {\sc scc}, denoted by $M$ and defined, for every $p\in\mathcal{P}$, by\footnote{Fishburn (1977) presents the equivalent definition
\[
\begin{array}{l}
M(p)=\underset{x\in N}{\mathrm{argmax}} \underset{y\in N\setminus\{x\}}{\min} |\{i\in H: x>_{p_i}y\}|.
\end{array}
\] }
\[
M(p)=\underset{x\in N}{\mathrm{argmin}} \underset{y\in N\setminus\{x\}}{\max} |\{i\in H: y>_{p_i}x\}|.
\]
According to the above definition, $M(p)$ is then the set of those alternatives which minimize the greatest pairwise defeat, with respect to the individual preferences described by $p$.
However, the outcomes of the Minimax {\sc scc} admit an alternative interpretation in terms of majority thresholds.

Given $\mu \in \mathbb{N}\cap(h/2,h]$,  define, for every $p\in\mathcal{P}$, the set
\[
D_{\mu}(p)=\{x\in N: \forall y\in N,\,|\{i\in H: y>_{p_i}x\}|< \mu\}.
\]
 Thus, an alternative $x$ belongs to $D_{\mu}(p)$ if and only if it cannot be found another alternative which is preferred to $x$ by at least $\mu$ individuals, according to the preference profile $p$.
Note that the set $D_{\mu}(p)$ corresponds to the set of $\mu$-majority equilibria associated with $p$ as defined by Greenberg (1979) in the more general setting where individual preferences are represented via complete and transitive relations.
Observe that if $\mu\le \mu'$, then $D_{\mu}(p)\subseteq D_{\mu'}(p)$ for all $p\in\mathcal{P}$.
Moreover, as an immediate consequence of Corollary 3 in Greenberg (1979) and its proof,  for every $\mu\in \mathbb{N}\cap (h/2,h]$, we have that
\begin{equation}\label{Greenberg}
D_{\mu}(p)\neq \varnothing\mbox{ for all }p\in\mathcal{P}\mbox{ if and only if }\mu>\frac{n-1}{n}h.
\end{equation}
Since $h \in \mathbb{N}\cap(h/2,h]$ and $h>\frac{n-1}{n}h$, it is well defined the {\it Greenberg majority threshold} given by
\[
\mu_{G}=\min\left\{m\in \mathbb{N}\cap (h/2,h]: m>\frac{n-1}{n}h\right\}.
\]
For every $p\in\mathcal{P}$, we consider the integer
\[
\mu(p)=\min\{\mu \in \mathbb{N}\cap(h/2,h]: D_{\mu}(p)\neq \varnothing\}.
\]
Note that, since \eqref{Greenberg} implies $D_{\mu_G}(p)\neq \varnothing$, we have that $\mu(p)$ is well defined and
$ \mu_0\leq \mu(p)\leq \mu_G$.

We can now prove the following proposition.

\begin{proposition}\label{Kramer}  For every $p\in \mathcal{P}$, $M(p)=D_{\mu(p)}(p)$.
\end{proposition}
\begin{proof}
We show first that $M(p)\subseteq D_{\mu(p)}(p)$ proving that $N\setminus D_{\mu(p)}(p)\subseteq N\setminus M(p)$. Let $x_0\in N\setminus D_{\mu(p)}(p).$ Then there exists $y_0\in N\setminus\{x_0\}$ such that $ |\{i\in H: y_0>_{p_i}x_0\}|\geq \mu(p)$.
 Picking now $x_1\in D_{\mu(p)}(p)$, we have that, for every $y\in N\setminus\{x_1\}$, $ |\{i\in H: y>_{p_i}x_1\}|\leq \mu(p)-1$. Thus,
\[
\underset{y\in N\setminus\{ x_1\}}{\max} |\{i\in H: y>_{p_i}x_1\}|\leq \mu(p)-1<|\{i\in H: y_0>_{p_i}x_0\}|\le \underset{y\in N\setminus\{x_0\}}{\max} |\{i\in H: y>_{p_i}x_0\}|,
\]
 which says $x_0\notin M(p)$.

We next show that $D_{\mu(p)}(p)\subseteq M(p).$  Let $x_0\in D_{\mu(p)}(p)$. Then, we have that
\[
\underset{y\in N\setminus\{ x_0\}}{\max} |\{i\in H: y>_{p_i}x_0\}|\leq \mu(p)-1.
\]
 Assume, by contradiction, that there exists $x_1\in N$ such that
\[
\underset{y\in N\setminus\{x_1\}}{\max} |\{i\in H: y>_{p_i}x_1\}|<\underset{y\in N\setminus\{ x_0\}}{\max} |\{i\in H: y>_{p_i}x_0\}|.
\]
Then $x_0\neq x_1$ and, for every $y\in N\setminus\{ x_1\}$, we have $|\{i\in H: y>_{p_i}x_1\}|<\mu(p)-1$. If $\mu(p)-1>h/2$, that says $x_1\in D_{\mu(p)-1}(p)=\varnothing$ and the contradiction is found. Assume instead that $\mu(p)-1\leq h/2$. Then $\mu(p)=\mu_0=\left\lceil \frac{h+1}{2}\right\rceil\leq \frac{h+2}{2}$ and since $|\{i\in H: x_0>_{p_i}x_1\}|\leq\mu_0-2,$ we get $|\{i\in H: x_1>_{p_i}x_0\}|\geq h-\mu_0+2.$ Now we observe that, due to $\mu_0\leq \frac{h+2}{2}$, we have $h-\mu_0+2\geq \mu_0$, against $x_0\in D_{\mu_0}(p)$.
\end{proof}

Define the sets
\[
\begin{array}{l}
T_1=\{(h,n)\in\mathbb{N}_\diamond^2:h\le 3\}\cup\{(h,n)\in\mathbb{N}_\diamond^2: n\le 3\}\cup\{(4,4), (5,4), (7,4), (5,5)\},\\
\vspace{-1mm}\\
T_2=\{(h,n)\in\mathbb{N}_\diamond^2: h=2\}\cup\{(h,n)\in\mathbb{N}_\diamond^2: n\le 3\}\cup\{(4,4)\},\\
\vspace{-1mm}\\
T_3=\{(h,n)\in\mathbb{N}_\diamond^2: n= 2\}\cup\{(3,3)\}.
\end{array}
\]
and note that $T_3\subsetneq T_2\subsetneq T_1$.
We can now state the main result of the paper. Its proof is technical and will be presented in Section \ref{proofs}. We stress that it relies on Proposition \ref{Kramer} and the use of language and methods of graph theory (Sections \ref{graph-sec} and \ref{majority graphs}).

\begin{theorem}\label{main-new}
Let $j\in\{1,2,3\}$. Then, $M\in\mathfrak{C}^j$ if and only if $(h,n)\in T_j$.
\end{theorem}

\section{The Borda and Copeland {\sc scc}s}

In this section we show that, differently from the case of the Minimax {\sc scc}, the analysis of the reversal bias is easy for the  Borda and Copeland {\sc scc}s.
Those {\sc scc}s are respectively denoted by $Bor$ and  $Cop$, and defined\footnote{With Borda {\sc scc} we mean the well-known Borda count. The definition of the Copeland {\sc scc} can be
found, for instance, in Fishburn (1977).}, for every $p\in\mathcal{P}$, as
\[
\begin{array}{l}
Bor(p)=\underset{x\in N}{\mathrm{argmax}} \sum_{i=1}^h \big(n-\mathrm{rank}_{p_i}(x)\big),\\
\\
Cop(p)=\underset{x\in N}{\mathrm{argmax}}\, \big(|\{y\in N: x>_{\mu_0}^p y\}|-|\{y\in N: y>_{\mu_0}^p x\}|\big).
\end{array}
\]
The following results show that they are immune to the reversal bias of type 3.
\begin{proposition}\label{sym3} $Bor,\ Cop\in \mathfrak{C}^3$.
\end{proposition}

\begin{proof}  We start considering the Borda {\sc scc}. We need to show that, for every $p\in\mathcal{P}$, $Bor(p)\cap Bor(p^r)\neq\varnothing$ implies $Bor(p)=N$.
Fix then $p\in\mathcal{P}$ and $x_0\in N$ such that  $x_0\in Bor(p)\cap Bor(p^r)$. Let $f$, $g$ and $u$ be the functions from $N$ to $\mathbb{R}$
defined, for every $x\in N$, by
\[
f(x)=\sum_{i=1}^h \big(n-\mathrm{rank}_{p_i}(x)\big),\quad g(x)=\sum_{i=1}^h \big(n-\mathrm{rank}_{p^r_i}(x)\big),\quad u(x)=\sum_{i=1}^h \mathrm{rank}_{p_i}(x).
\]
Note that
\[
Bor(p)=\underset{x\in N}{\mathrm{argmax}} \,f(x), \quad
Bor(p^r)=\underset{x\in N}{\mathrm{argmax}}\, g(x),
\]
and that, for every $x\in N$, $f(x)=hn-u(x)$ and $g(x)=u(x)-h$, due to the fact that  $\mathrm{rank}_{p^r}(x)=n+1-\mathrm{rank}_{p}(x).$
Then $x_0$ realises  both the minimum and the maximum of $u$, so that $u$ is constant. It follows that $f$ is constant too and therefore $Bor(p)=N.$

We next consider the Copeland {\sc scc}. We need to show that, for every $p\in\mathcal{P}$, $Cop(p)\cap Cop(p^r)\neq\varnothing$ implies $Cop(p)=N$.
Fix then $p\in\mathcal{P}$ and $x_0\in N$ such that  $x_0\in Cop(p)\cap Cop(p^r)$. Let $f$ and $g$ be functions from $N$ to $\mathbb{R}$
defined, for every $x\in N$, by
\[
f(x)= |\{y\in N: x>_{\mu_0}^p y\}|-|\{y\in N: y>_{\mu_0}^p x\}|, \quad g(x)= |\{y\in N: x>_{\mu_0}^{p^r} y\}|-|\{y\in N: y>_{\mu_0}^{p^r} x\}|.
\]
Note that
\[
Cop(p)=\underset{x\in N}{\mathrm{argmax}} \,f(x), \quad
Cop(p^r)=\underset{x\in N}{\mathrm{argmax}}\, g(x).
\]
Moreover, since  $x>_{p^r_i}y$ is equivalent to $y>_{p_i}x$ for all $x,y\in N$ and $i\in H$, we have that, for every $x\in N$, $g(x)=-f(x)$.
 Then $x_0$ realises  both the minimum and the maximum of $f$. It follows that $f$ is constant and therefore $Cop(p)=N.$
\end{proof}

Next corollaries show how the results proved about the reversal bias of $M$, $Bor$ and $Cop$ can be used to establish conditions on the number of individuals and alternatives that are necessary and sufficient to have the equalities $M=Bor$ and $M=Cop$.

\begin{corollary}\label{confronto1}
 $M=Bor$ if and only if $n=2$.
\end{corollary}
\begin{proof}If $n=2$, then we surely have $M=Bor$. Assume now that $n\ge 3$. If $(h,n)\not \in T_3$, then, by Theorem \ref{main-new} and Proposition \ref{sym3},  $M\not \in \mathfrak{C}^3$  and $Bor\in \mathfrak{C}^3$ so that $M\neq Bor$. If  $(h,n)\in T_3$, then
$(h,n)=(3,3)$ and we still have $M\neq Bor$ since the two {\sc scc}s differ, for instance, on the preference profile
\[
p=\left[
\begin{array}{cccccc}
1&1&2\\
2&2&3\\
3&3&1\\
\end{array}
\right]
\]
\end{proof}

\begin{corollary}\label{confronto2}
$M=Cop$ if and only if $(h,n)\in T_3$.
\end{corollary}
\begin{proof}
If $n=2$, then we surely have $M=Cop$. Assume now that $n\ge 3$. If $(h,n)\not \in T_3$, then, by Theorem \ref{main-new} and Proposition \ref{sym3},  $M\not \in \mathfrak{C}^3$  and $Cop\in \mathfrak{C}^3$ so that $M\neq Cop$. If $(h,n)\in T_3$, then
$(h,n)=(3,3)$. Consider $p\in\mathcal{P}$. Since both $M$ and $Cop$ are neutral, without loss of generality, we can assume that $p_1=[1,2,3]^T$. Recalling that both $M$ and $Cop$ satisfy the Condorcet principle, they surely coincide when a Condorcet winner exists. Thus, we can assume that the alternatives ranked first are different among themselves. Since both $M$ and $Cop$ are anonymous, we can assume that $p_2(1)=2$ and $p_3(1)=3$. That leaves just four cases to treat.
By a case by case computation on those, it can be finally proved that $M=Cop$.
\end{proof}

\section{Proof of Theorem \ref{main-new}}\label{proofs}

From  Proposition \ref{Kramer} we immediately have that Theorem \ref{main-new} is implied by the following three propositions. Their tricky proofs, based on graph theory, are presented in the Sections \ref{TM1-sec}, \ref{TM2-sec} and \ref{TM3-sec}.

\begin{proposition}\label{TM1}
There exist $p\in\mathcal{P}$ and $x\in N$ such that $D_{\mu(p)}(p)=D_{\mu(p^r)}(p^r)=\{x\}$ if and only if $(h,n)\in \mathbb{N}_\diamond^2\setminus T_1$.
\end{proposition}

\begin{proposition}\label{TM2}
There exist $p\in\mathcal{P}$ and $x\in N$ such that $D_{\mu(p)}(p)=\{x\}\subseteq D_{\mu(p^r)}(p^r)$ if and only if $(h,n)\in \mathbb{N}_\diamond^2\setminus T_2$.	
\end{proposition}

\begin{proposition}\label{TM3}
There exists $p\in\mathcal{P}$ such that $D_{\mu(p)}(p)\neq N$ and $D_{\mu(p)}(p)\cap D_{\mu(p^r)}(p^r)\neq\varnothing$ if and only if $(h,n)\in \mathbb{N}_\diamond^2\setminus T_3$.
\end{proposition}

\subsection{Graphs}\label{graph-sec}

In this section, we recall some basic facts and notation from graph theory, which we are going to use in the sequel\footnote{All unexplained notation is standard. See, for instance, Diestel (2010).}. All the considered graphs are directed.
A {\it graph} is a pair $(V,A)$, where $V$ is a nonempty set called vertex set and $A$ is a subset of $\{(x,y)\in V^2: x\neq y\}$ called arc set. Note that if $\Gamma=(V,A)$ is a graph and $|V|=1$, then $A=\varnothing$.
Given two graphs $\Gamma_1=(V_1,A_1)$ and $\Gamma_2=(V_2,A_2)$, we say that $\Gamma_2$ is a {\it subgraph} of $\Gamma_1$ if $V_2\subseteq V_1$ and $A_2\subseteq A_1$. If $\Gamma_2$ is a subgraph of $\Gamma_1$, we write $\Gamma_2\leq \Gamma_1$.

Let now  $\Gamma=(V,A)$ be a graph. $\Gamma$ is called {\it complete} if for every $x,y\in V$ with $x\neq y$, we have $(x,y)\in A$ or $(y,x)\in A$.
We say that $x\in V$ is {\it maximal} [{\it minimal}] for $\Gamma$ if there exists no $y\in V$ such that $(y,x)\in A$  $[(x,y)\in A]$. We denote by $\max(\Gamma)$ [$\min(\Gamma)$] the set of maximal [minimal] vertices for $\Gamma$.
Note that those sets may be empty. We say that $x\in V$ is a {\it maximum} [{\it minimum}] of $\Gamma$ if, for every $y\in V\setminus \{x\}$, we have that $(x,y)\in A$ $[(y,x)\in A]$\footnote{Note that if $x$ is a maximum [minimum] of $\Gamma$ it is not necessarily maximal [minimal] for $\Gamma$. In fact, given $\Gamma=(\{1,2\},\{(1,2),(2,1)\})$, we have that  $1$ and $2$ are both a maximum [minimum] but none of them is maximal [minimal].}.  We denote by $\mathrm{Max}(\Gamma)$ [$\mathrm{Min}(\Gamma)$] the set of maxima [minima] of $\Gamma$. We say that $x\in V$ is {\it isolated} in $\Gamma$ if, for every $y\in V\setminus\{x\}$, $(x,y), (y,x)\not\in A$. We denote by $\mathrm{I}(\Gamma)$ the set of the isolated vertices of $\Gamma$.
It is useful to note that
\begin{equation}\label{intersection}
\max(\Gamma)\cap \min (\Gamma)=\mathrm{I}(\Gamma).
\end{equation}
Note also that if $x\in\mathrm{Max}(\Gamma)\cup \mathrm{Min}(\Gamma)$ and $|V|\geq 2$, then $x\not\in \mathrm{I}(\Gamma).$

$\Gamma$ is said to be {\it connected} if, for every $x, y\in V$ with $x\neq y$, there exist $k\ge 2$ and an ordered sequence $x_1,\ldots, x_k$ of distinct elements of $V$ such that $x_1=x$, $x_k=y,$ and, for every $j\in\{1,\ldots,k-1\}$, $(x_j,x_{j+1})\in A$ or $(x_{j+1},x_{j})\in A$. Note that if $\Gamma$ has a maximum [minimum], then $\Gamma$ is connected.
It is well known that there exist a uniquely determined $c\in \mathbb{N}$ and  connected subgraphs
$\Gamma_1=(V_1,A_1),\ldots,\Gamma_c=(V_c,A_c)$ of $\Gamma$ such that $\cup_{i=1}^cV_i=V$, $\cup_{i=1}^cA_i=A$, and for every $i,j\in\{1,\ldots,c\}$ with $i\neq j$,  $V_i\cap V_j=A_i\cap A_j=\varnothing$. Those subgraphs $\Gamma_1,\dots, \Gamma_c$ are called the {\it connected components} of $\Gamma$. They are maximal among the connected subgraphs of $\Gamma$, that is, if $\Gamma'\leq \Gamma$ is connected and $\Gamma'\geq \Gamma_i$ for some $i\in\{1,\ldots,c\},$ then $\Gamma'= \Gamma_i.$ In particular, for every $i\in\{1,\ldots,c\}$, $x\in V_i$ and $y\in V\setminus V_i$ imply $(x,y),(y,x)\notin A$; $x,y\in V_i$ and $(x,y)\in A$ imply $(x,y)\in A_i.$ Note that $x\in N$ is isolated in $\Gamma$ if and only if the connected component of $\Gamma$ containing $x$ is $(\{x\},\varnothing).$
Given $l\ge 2$, $\Gamma$ is said to be a {\it $l$-cycle} if $|V|=l$ and there exists an ordered sequence $x_1,\ldots, x_l$ of the elements of $V$ such that, once defined  $x_{l+1}=x_1$, we have that $A=\{(x_j,x_{j+1})\ :1\leq j\leq l\}$. $\Gamma$ is said to be  a {\it cycle} if it is a $l$-cycle for some $l\ge 2$.
Fixed $l\ge 2$, $\Gamma$ is said to be {\it $l$-cyclic} if there exists a $l$-cycle $\Gamma_1\le \Gamma$, and {\it $l$-acyclic} otherwise.
$\Gamma$ is said to be {\it acyclic} if it is $l$-acyclic for all $l\ge 2$. Note that if $|V|=1$, then $\Gamma$ is acyclic.

\subsection{Majority graphs and their properties}\label{majority graphs}

Let $p\in \mathcal{P}$ and $\mu\in \mathbb{N}\cap (h/2,h].$
In a natural way, we associate with the relation on $N$ given by $\Sigma_{\mu}(p)=\{(x,y)\in N\times N: x>_\mu^p y\}$, the graph $\Gamma_{\mu}(p)=(N,\Sigma_{\mu}(p))$, called the {\it $\mu$-majority graph} of  $p$. Note that,  if $\mu, \mu'\in \mathbb{N}\cap (h/2,h]$ with $ \mu'\le \mu$, then $\Gamma_{\mu}(p)\leq \Gamma_{\mu'}(p).$ In particular, $\Gamma_{\mu}(p)\leq \Gamma_{\mu_0}(p)$ holds for all $\mu\in \mathbb{N}\cap (h/2,h].$ The concept of majority graph has been considered by many authors essentially in relation to the case when $h$ is odd and $\mu=\mu_0=\frac{h+1}{2}$ (see, for instance, Miller (1977)).
For the purpose of our paper that case is interesting because $\Gamma_{\frac{h+1}{2}}(p)$ is complete (see Lemma \ref{minimalmu-i}), but we are not focussed only on that particular majority graph.

The properties of the relation $\Sigma_{\mu}(p)$ translates easily into graph theoretical properties for $\Gamma_{\mu}(p)$. Moreover, considering $\Gamma_{\mu}(p)$ we gain the advantage of using concepts like $l$-acyclicity and connectedness which typically belongs to graph theory. That gives very soon
a better comprehension of the  sets $D_{\mu}(p)$ and  $D_{\mu(p)}(p)=M(p)$.

\begin{lemma}\label{intersection2}
Let $\mu\in \mathbb{N}\cap (h/2,h]$ and $p\in \mathcal{P}$. Then $D_\mu(p)=\max(\Gamma_\mu(p))=\min(\Gamma_\mu(p^{r}))$. Moreover  $ D_{\mu}(p)\cap D_{\mu}(p^{r})=\mathrm{I}(\Gamma_{\mu}(p))=\mathrm{I}(\Gamma_{\mu}(p^{r})).$
\end{lemma}
\begin{proof}
The equalities
$D_\mu(p)=\max(\Gamma_\mu(p))=\min(\Gamma_\mu(p^{r}))$ follow from the definitions of $D_\mu(p)$ and $p^{r}$. As a consequence, since $(p^r)^r=p$, we also have
$D_{\mu}(p^{r})=\max(\Gamma_{\mu}(p^{r}))=\min (\Gamma_{\mu}(p))$, so that \eqref{intersection} completes the proof.
\end{proof}

\begin{lemma}\label{maximum}
Let $\mu\in \mathbb{N}\cap (h/2,h]$ and $p\in \mathcal{P}$.  Then $\Gamma_{\mu}(p)$ is $2$-acyclic and $\Gamma_{\mu}(p)$ has at most one maximum. Moreover, if $\Gamma_{\mu}(p)$ has a maximum $x\in N$, then $D_{\mu}(p)= \mathrm{Max}(\Gamma_{\mu}(p))=\{x\}$.
\end{lemma}
\begin{proof}
The $2$-acyclicity follows immediately from $\mu>h/2$.

Assume, by contradiction, that there exist distinct $ x, y\in \mathrm{Max}(\Gamma_{\mu}(p))$. Then $(x,y)\in \Sigma_\mu(p)$ and $(y,x)\in \Sigma_\mu(p)$, so that $\Gamma_1=(\{x,y\},\{(x,y),(y,x)\})\le \Gamma_{\mu}(p)$. Since $\Gamma_1$ is a $2$-cycle, that contradicts the fact that $\Gamma_{\mu}(p)$ is $2$-acyclic.

Assume now that there exists $x\in \mathrm{Max}(\Gamma_{\mu}(p))$ and show that $x\in \max(\Gamma_{\mu}(p))=D_\mu(p)$. By contradiction, let $x\not\in \max(\Gamma)$. Then there is $y\in N$ such that $(y,x)\in \Sigma_\mu(p)$. Since also $(x,y)\in \Sigma_\mu(p)$, the $2$-cycle $\Gamma_1=(\{x,y\},\{(x,y),(y,x)\})$ is a subgraph of $\Gamma_\mu(p)$ and the contradiction is found. We complete the proof simply noticing that, being $x$ a maximum of $\Gamma_\mu(p)$, for every $y\in N\setminus \{x\}$, we have that $(x,y)\in \Sigma_\mu(p)$ so that $y\not\in \max(\Gamma_\mu(p))$.
\end{proof}

\begin{lemma}\label{ac} Let $\mu\in \mathbb{N}\cap (h/2,h]$. Then $\Gamma_{\mu}(p)$ is acyclic for all $p\in \mathcal{P}$ if and only if $\mu\geq \mu_G.$ In particular $\Gamma_{h}(p)$ is acyclic for all $p\in \mathcal{P}.$
\end{lemma}
\begin{proof}
It is an immediate consequence of Proposition $6$ and $7$ in Bubboloni and Gori (2014).
\end{proof}

\begin{lemma}\label{minimalmu-i} If $h$ is odd, then, for every $p\in \mathcal{P}$, $\Gamma_{\mu_0}(p)$  is complete and $\mathrm{I}(\Gamma_{\mu_0}(p))=\varnothing$. Moreover, if $D_{\mu_0}(p)\neq \varnothing$ then  $\mu(p)=\mu_0$, $\Gamma_{\mu_0}(p)$ admits maximum $x\in N$ and $D_{\mu_0}(p)=\{x\}$.
\end{lemma}

\begin{proof} Let us fix $p\in\mathcal{P}$ and note that, being $h$ odd, we have $\mu_0=\frac{h+1}{2}$.
Assume now, by contradiction, that there exist $x,y\in N$ with $x\not>^p_{\mu_0}y$ and $y\not>^p_{\mu_0} x$. Then, we get the impossible relation
\[
h=|\{i\in H : x>_{p_i}y\}|+|\{i\in H : y>_{p_i}x\}|\leq \mu_0 -1+\mu_0 -1=2\left(\frac{h+1}{2}\right)-2=h-1.
\]
 Thus  $\Gamma_{\mu_0}(p)$  is complete and, as an immediate consequence, $\mathrm{I}(\Gamma_{\mu_0}(p))=\varnothing.$

In order to prove the second part, assume that $D_{\mu_0}(p)\neq \varnothing$. Then $\mu(p)\leq\mu_0 $ and so  $\mu(p)=\mu_0$.
Next, pick $x\in D_{\mu_0}(p)$. Since $\Gamma_{\mu_0}(p)$ is complete, then we have $x>_{\mu_0}^p y$ for all $y\in N\setminus\{x\},$ that is, $x$ is a maximum in $\Gamma_{\mu_0}(p)$. Then, by Lemma \ref{maximum}, $D_{\mu_0}(p)=\{x\}$.
\end{proof}

Let us denote by $\mathcal{C}(\Gamma_{\mu}(p))$ the set of the connected components of $\Gamma_{\mu}(p)$ and define $\mathcal{A}(\Gamma_{\mu}(p))=\{\Gamma\in\mathcal{C}(\Gamma_{\mu}(p)) : \Gamma \ \hbox{is acyclic} \}$. We are ready for a key proposition giving a lower bound for $|D_{\mu}(p)|$ and leading to some interesting consequences.

\begin{proposition}\label{one} Let $\mu\in \mathbb{N}\cap (h/2,h]$ and  $p\in \mathcal{P}$. Then
\[
D_{\mu}(p)
=\underset{\Gamma\in\mathcal{C}(\Gamma_{\mu}(p))}{\bigcup}\max(\Gamma)\supseteq\underset{\Gamma\in\mathcal{A}(\Gamma_{\mu}(p))}{\bigcup}\max(\Gamma) \supseteq\mathrm{I}(\Gamma_{\mu}(p)),
\]
 and
 \[
 |D_{\mu}(p)|
=\underset{\Gamma\in\mathcal{C}(\Gamma_{\mu}(p))}{\sum}|\max(\Gamma)|\geq |\mathcal{A}(\Gamma_{\mu}(p))|\geq |\mathrm{I}(\Gamma_{\mu}(p))|.
\]
\end{proposition}

\begin{proof} Let $\Gamma=(V,A)\in \mathcal{C}(\Gamma_{\mu}(p)).$ Then $\Gamma\leq \Gamma_{\mu}(p)$, so that $V\subseteq N$ and $A\subseteq \Sigma_{\mu}(p)$. Since $\Gamma$ is a connected component of $\Gamma_{\mu}(p)$, we have that, for every $x\in V$ and $y\in N\setminus V$,
$y\not>^p_{\mu} x$.
This immediately gives that each $x\in \max(\Gamma)$ belongs to $D_{\mu}(p),$ so that $D_{\mu}(p)
\supseteq\underset{\Gamma\in\mathcal{C}(\Gamma_{\mu}(p))}{\bigcup}\max(\Gamma).$ The other inclusion is trivial and thus $D_{\mu}(p)
=\underset{\Gamma\in\mathcal{C}(\Gamma_{\mu}(p))}{\bigcup}\max(\Gamma).$ Since $\mathcal{A}(\Gamma_{\mu}(p))\subseteq \mathcal{C}(\Gamma_{\mu}(p))$ and, for every $x\in \mathrm{I}(\Gamma_\mu(p))$, $(\{x\},\varnothing)\in \mathcal{A}(\Gamma_\mu(p))$,
we also get
\[
\underset{\Gamma\in\mathcal{C}(\Gamma_{\mu}(p))}{\bigcup}\max(\Gamma)\supseteq\underset{\Gamma\in\mathcal{A}(\Gamma_{\mu}(p))}{\bigcup}\max(\Gamma)\supseteq\mathrm{I}(\Gamma_{\mu}(p)).
\]
In particular, since there is no overlap between vertices of different connected components, we  deduce $|D_{\mu}(p)|
=\underset{\Gamma\in\mathcal{C}(\Gamma_{\mu}(p))}{\sum}|\max(\Gamma)|$. We complete the proof showing that for every $\Gamma\in \mathcal{A}(\Gamma_{\mu}(p))$, we have $\max(\Gamma)\neq \varnothing.$
 Pick $x_1\in V$. If  $y\not >^p_{\mu} x_1$ for all $y\in V$,
then we have $x_1\in \max(\Gamma) $ and we have finished. Assume instead there exists $x_2\in V$  with $x_2>^p_{\mu} x_1.$ Obviously, we have $x_2\neq x_1.$
Then, repeat the argument for $x_2$. Since the set $N$ is finite and $\Gamma$ contains no cycle, in a finite number $k\leq n$ of steps, we obtain an element $x_k\in \max(\Gamma).$
\end{proof}

\begin{corollary}\label{muacyclic} Let $\mu \in \mathbb{N}\cap (h/2,h]$ and $p\in\mathcal{P}.$ If $\Gamma_{\mu}(p)$ admits at least an acyclic connected component, then $\mu(p)\leq\mu.$
\end{corollary}
\begin{proof} By Proposition \ref{one}, we have $|D_{\mu}(p)|\geq 1,$ so that $D_{\mu}(p)\neq \varnothing$.
\end{proof}

\begin{corollary}\label{cor-connected}
 Let $\mu \in \mathbb{N}\cap (h/2,h]$ and $p\in\mathcal{P}.$ If $\Gamma_\mu(p)$ is acyclic and $D_\mu(p)$ is a singleton, then $\Gamma_\mu(p)$ is connected.
\end{corollary}

\begin{proof}
Since $\Gamma_\mu(p)$ is acyclic, we have that $\mathcal{C}(\Gamma_\mu(p))=\mathcal{A}(\Gamma_\mu(p))$. Then, using Proposition \ref{one}, we get $1=|D_\mu(p)|\ge |\mathcal{C}(\Gamma_\mu(p))|\ge 1$. That implies $|\mathcal{C}(\Gamma_\mu(p))|= 1$, that is, $\Gamma_\mu(p)$ is connected.
\end{proof}

\begin{lemma}\label{intersection3}
Let $p\in \mathcal{P}$ such that $\mu(p^{r})\leq \mu(p)$. Then:
\begin{itemize}\item[(i)]
 $ D_{\mu(p)}(p)\cap D_{\mu(p^{r})}(p^{r})\subseteq\mathrm{I}(\Gamma_{\mu(p)}(p)).$ In particular, if $\Gamma_{\mu(p)}(p)$ is connected, then
 $ D_{\mu(p)}(p)\cap D_{\mu(p^{r})}(p^{r})=\varnothing.$
 \item[(ii)] If $|D_{\mu(p)}(p)|=1$ and $\Gamma_{\mu(p)}(p)$ is acyclic, then $ D_{\mu(p)}(p)\cap D_{\mu(p^{r})}(p^{r})=\varnothing.$
 \end{itemize}
\end{lemma}

\begin{proof}$(i)$ From $\mu(p^{r})\leq \mu(p)$  we get $D_{\mu(p^r)}(p^r)\subseteq D_{\mu(p)}(p^r)$ and thus, by Lemma \ref{intersection2}, we deduce $ D_{\mu(p)}(p)\cap D_{\mu(p^{r})}(p^{r})\subseteq D_{\mu(p)}(p)\cap D_{\mu(p)}(p^{r})=\mathrm{I}(\Gamma_{\mu(p)}(p)).$
If $\Gamma_{\mu(p)}(p)$ is connected, then $\mathrm{I}(\Gamma_{\mu(p)}(p))$ is empty and thus also $ D_{\mu(p)}(p)\cap D_{\mu(p^{r})}(p^{r})=\varnothing.$

$(ii)$ By Corollary \ref{cor-connected}, (i) applies giving $ D_{\mu(p)}(p)\cap D_{\mu(p^{r})}(p^{r})=\varnothing.$
\end{proof}

\begin{corollary}\label{muacyclicnew} Let $\mu \in \mathbb{N}\cap (h/2,h]$ and $p\in\mathcal{P}.$ If $\Gamma_{\mu}(p)$ is acyclic, then, for every  $x\in N$, we do not have
$D_{\mu}(p)=D_{\mu}(p^{r})=\{x\}.$

\end{corollary}
\begin{proof}  Assume by contradiction that $D_{\mu_0}(p)=D_{\mu_0}(p^{r})=\{x\},$ for some $x\in N$. Then, by  Lemma \ref{intersection2}, we have that $x$ is isolated in $\Gamma_{\mu}(p)$. On the other hand, by Corollary \ref{cor-connected}, $\Gamma_{\mu}(p)$ is connected so that its only vertex is $x$, against $n\geq 2.$
\end{proof}

\begin{lemma}\label{four}
Let $p\in \mathcal{P}$ and assume that both $\Gamma_{\mu(p)}(p)$ and $\Gamma_{\mu(p^{r})}(p^{r})$ admit an acyclic connected component. Then:
\begin{itemize}
\item[(i)] $\mu(p)=\mu(p^r)$.
\item[(ii)] If $\Gamma_{\mu(p)}(p)$ is connected, then $ D_{\mu(p)}(p)\cap D_{\mu(p^{r})}(p^{r})=\varnothing.$
\item[(iii)] If $\Gamma_{\mu(p)}(p)$ is acyclic, then
there exists no $x\in N$ such that
$D_{\mu(p)}(p)=  D_{\mu(p^{r})}(p^{r})=\{x\}.$
\end{itemize}
\end{lemma}

\begin{proof} $(i)$ The fact that $\Gamma_{\mu(p)}(p)$ admits an acyclic connected component implies that also $\Gamma_{\mu(p)}(p^r)$ admits an acyclic connected component and therefore Corollary \ref{muacyclic} gives $\mu(p^r)\leq \mu(p)$. The same argument applied to $\Gamma_{\mu(p^{r})}(p^{r})$ gives  $\mu(p)\leq \mu(p^r).$

$(ii)$ Assume that $\Gamma_{\mu(p)}(p)$ is connected. Since, by $(i)$,  we have that $\mu(p)=\mu(p^r)$, then Lemma \ref{intersection3} applies, giving $ D_{\mu(p)}(p)\cap D_{\mu(p^{r})}(p^{r})=\varnothing.$

$(iii)$ Let $\Gamma_{\mu(p)}(p)$  be acyclic and assume, by contradiction, that there exists $x\in N$ such that $D_{\mu(p)}(p)=  D_{\mu(p^{r})}(p^{r})=\{x\}$.
By Corollary \ref{cor-connected}, $\Gamma_{\mu(p)}(p)$ is connected  and thus, by $(ii)$, we have that $ D_{\mu(p)}(p)\cap D_{\mu(p^{r})}(p^{r})=\varnothing,$ a contradiction.
\end{proof}

\begin{lemma}\label{minimalmu-ii} If $h$ is odd, then, for every $p\in \mathcal{P}$, $\mu(p)=\mu_{0}$ and $D_{\mu(p)}(p)=D_{\mu(p^{r})}(p^{r})$ imply $\mu(p^{r})>\mu_0$.
\end{lemma}

\begin{proof}Let $p\in\mathcal{P}$ and assume that $\mu(p)=\mu_0$ and  $D_{\mu(p)}(p)=D_{\mu(p^{r})}(p^{r})$. Then $D_{\mu_0}(p)\neq \varnothing$ and, using Lemma \ref{minimalmu-i},
$\Gamma_{\mu(p)}(p)$ has a maximum $x\in N$ and $ D_{\mu(p)}(p)=D_{\mu(p^{r})}(p^{r})=\{x\}$. Assume by contradiction that $\mu(p^{r})=\mu_0.$  By Lemma \ref{intersection2}, we get that $x$ is isolated in $\Gamma_{\mu(p)}(p)$, against the fact that $x$ is the maximum of $\Gamma_{\mu(p)}(p)$.
\end{proof}

\begin{corollary}\label{mupmumin} Let $p\in \mathcal{P}$ such that $\mu(p)=\mu_{0}$. If $\Gamma_{\mu(p)}(p)$ is acyclic, then there exists no $x\in N$ such that $D_{\mu(p)}(p)=  D_{\mu(p^{r})}(p^{r})=\{x\}$.
\end{corollary}
\begin{proof} The acyclicity of $\Gamma_{\mu(p)}(p)$  implies that of $\Gamma_{\mu(p)}(p^{r})$, so that, by Corollary \ref{muacyclic}, we have $\mu_0\leq\mu(p^{r})\leq \mu(p)=\mu_0$. It follows that $\mu(p^{r})=\mu(p)=\mu_0$ and Corollary \ref{muacyclicnew} applies.
\end{proof}

Due to the previous results, it is important to understand which conditions guarantee the acyclicity of $\Gamma_{\mu(p)}(p)$.
By Lemma \ref{maximum}, we know that, for every $\mu \in \mathbb{N}\cap (h/2,h]$,  $\Gamma_{\mu}(p)$ is $2$-acyclic. Anyway, it can admit $l$-cycles for some $l\ge 3$. We explore this possibility through Propositions 6 and 7 in Bubboloni and Gori (2014).

\begin{proposition}\label{green} Let  $\mu \in \mathbb{N}\cap (h/2,h]$ and $l\in\mathbb{N}\cap [2,n]$. Then there exists $p\in \mathcal{P}$ such that $\Gamma_{\mu}(p)$ is $l$-cyclic if and only if $\mu\leq \frac{l-1}{l}h.$
\end{proposition}

\begin{proof}  Consider $\mu>\frac{l-1}{l}h$ and assume by contradiction that there exists $p\in \mathcal{P}$ and an $l$-cycle $\Gamma\le \Gamma_{\mu}(p)$ with vertex set $V$. Then $V\subseteq N$ and $|V|=l\leq n$.
Consider the preference profile $p'$ on the set of $l$ alternatives $V$ obtained from $p$ eliminating (if any) those entries in $N\setminus V.$
By Proposition 6 in Bubboloni and Gori (2014), we have that
$\Gamma_{\mu}(p')$ is acyclic, against the fact that $\Gamma\le \Gamma_{\mu}(p').$

Let now $\mu\leq \frac{l-1}{l}h$ and let $V\subseteq N$ with $|V|=l$. By Proposition 7 in Bubboloni and Gori (2014), there exists a preference profile $p'$ on the set of alternatives $V$ such that $\Gamma_{\mu}(p')$ contains an $l$-cycle $\Gamma$ whose set of vertices is $V$. Consider a preference profile $p$ on the set of alternatives $N,$ in which every individual $i\in H$ ranks in the first $l$ positions the alternatives in $V$ as $p_i'$ and those in $N\setminus V$ as she likes. Then $\Gamma\le \Gamma_{\mu}(p)$.
\end{proof}

Let us consider now
\[
\mu_{a}=\min\left\{m\in \mathbb{N}\cap (h/2,h]: m>\frac{n-2}{n-1}h\right\},
\]
and note that $\mu_a$ is well defined because $h\in \mathbb{N}\cap (h/2,h]$ and $h>\frac{n-2}{n-1}h$. Moreover, we have that $\mu_0\leq \mu_{a}\leq \mu_G$ and, when $n\in \{2,3\}$, $\mu_a=\mu_0$.

\begin{corollary}\label{soglia-green} Let $p\in\mathcal{P}$. If $\mu(p)\ge \mu_a,$ then $\Gamma_{\mu(p)}(p)$ is acyclic. In particular, for every  $n\in\{2,3\}$, $\Gamma_{\mu(p)}(p)$ is acyclic.
\end{corollary}
\begin{proof}
Consider $\Gamma_{\mu(p)}(p)$. It admits no $n$-cycle, because having such a cycle obviously implies the contradiction $D_{\mu(p)}(p)=\varnothing$. On the other hand, by Proposition \ref{green}, it does not have $l$-cycles for all $l\in\{2,\ldots, n-1\}$ because $\mu(p)> \frac{n-2}{n-1}h\geq \frac{l-1}{l}h.$ Finally note that, if $n\in\{2,3\}$, then $\mu(p)\ge \mu_0=\mu_a$.
\end{proof}
Due to the previous result, we call
$\mu_a$ the {\it acyclicity threshold}.

\begin{proposition}\label{blacktable-odd} If $n\geq 4$ and $h$ is odd and such that $h\geq \frac{3(n-1)}{n-3}$, then there exist $p\in\mathcal{P}$  such that $D_{\mu(p)}(p)=D_{\mu(p^{r})}(p^{r})=\{n\}.$
\end{proposition}

\begin{proof} First of all, note that $\mu_{0}=\frac{h+1}{2}$. Define then $\mu=\frac{h+3}{2}=\mu_0+1$ and $V=N\setminus\{n\}$. The assumption $h\geq \frac{3(n-1)}{n-3}$ is equivalent to $\mu\leq \frac{(n-1)-1}{n-1}h$ and thus,  by Proposition \ref{green}, there exists $p',$ a preference profile on the set of alternatives $V,$ such that $\Gamma_{\mu}(p')$ has an $(n-1)$-cycle $\Gamma$. We define now the preference profile $p\in\mathcal{P}$ defining, for every $i\in H$, the preference $p_i$ as follows.
If $i\leq \mu_0$, then let $p_i(1)=n$ and $p_i(j)=p'_i(j-1)$ for all $j\in\{2,\dots,n\}$; if $ \mu_0<i\leq h,$ then let $p_i(j)=p'_i(j)$ for all $j\in\{1,\dots,n-1\}$ and $p_i(n)=n.$ Note that in $p$, the alternative $n$ is ranked first $\mu_0$ times and last $h-\mu_0$ times. Thus, $n$ is a maximum in $\Gamma_{\mu_0}(p).$ By Lemma \ref{minimalmu-i}, we then get $\mu(p)=\mu_0$ and $D_{\mu(p)}(p)=\{n\}.$
 Moreover $\Gamma\leq \Gamma_{\mu}(p)$ so that also $\Gamma_{\mu}(p^{r})$ contains an $(n-1)$-cycle $\Gamma^r$, with inverted orientation,  whose vertex set is $V$.
That implies that $D_{\mu_0}(p^{r})=\varnothing$. Indeed, $n$ is not maximal in $\Gamma_ {\mu_0}(p^{r})$, being beaten $\mu_0$ times by any other alternative, and each alternative in $V$ is not maximal in  $\Gamma_ {\mu_0}(p^{r})$ because, due to the presence of the cycle $\Gamma^r,$ it is beaten $\mu>\mu_0$ times by a suitable alternative in $V$. Anyway $D_{\mu}(p^{r})=\{n\}$, because $n$ is isolated and thus maximal in $\Gamma_ {\mu}(p^{r}),$ by \eqref{intersection}; no other alternative is maximal because involved in $\Gamma^r$.
It follows that $\mu(p^{r})=\mu$ and $D_{\mu(p)}(p)=D_{\mu(p^{r})}(p^{r})=\{n\}.$
\end{proof}

\begin{proposition}\label{blacktable-even} If $n\geq 4$ and $h$ is even and such that $h\geq \frac{2(n-1)}{n-3}$, then there exist $p\in\mathcal{P}$  such that $D_{\mu(p)}(p)=D_{\mu(p^{r})}(p^{r})=\{n\}.$
\end{proposition}

\begin{proof} First of all, note that $\mu_{0}=\frac{h+2}{2}$ and define $V=N\setminus\{n\}$. The assumption $h\geq \frac{2(n-1)}{n-3}$ is equivalent to $\mu_0\leq \frac{(n-1)-1}{n-1}h$ and thus,  by Proposition \ref{green}, there exist a preference profile $p'$ on the set of alternatives $V$ such that $\Gamma_{\mu_0}(p')$ has an $(n-1)$-cycle $\Gamma$. We define now the preference profile $p\in\mathcal{P},$ defining, for every $i\in H$, the preference $p_i$ as follows.
If $i\leq \frac{h}{2}$, then  let $p_i(1)=n$ and $p_i(j)=p'_i(j-1)$ for all $j\in\{2,\dots,n\}$; if $ \frac{h}{2}<i\leq h,$ then let $p_i(j)=p'_i(j)$ for all $j\in\{1,\dots,n-1\}$ and $p_i(n)=n.$ Note that in $p$, the alternative $n$ is ranked first $\frac{h}{2}$ times  and last $\frac{h}{2}$ times. Thus,  by \eqref{intersection}, $n$ is isolated and maximal both in $\Gamma_{\mu_0}(p)$ and in
 $\Gamma_{\mu_0}(p^{r})$. Moreover, no further alternative is maximal in $\Gamma_{\mu_0}(p)$ because each element in $V$ is involved in the cycle $\Gamma\leq \Gamma_{\mu_0}(p)$. Since each cycle in $\Gamma_{\mu_0}(p)$ determines a cycle with inverted orientation in $\Gamma_{\mu_0}(p^{r})$, the same consideration holds for $\Gamma_{\mu_0}(p^{r})$, as well. Then, we  conclude that $\mu(p)=\mu(p^{r})=\mu_0$ and $D_{\mu(p)}(p)=D_{\mu(p^{r})}(p^{r})=\{n\}.$
\end{proof}

We conclude the section with a further lemma which is useful to manage the case with  three individuals and three alternatives.

\begin{lemma}\label{lemma33}
Let $(h,n)=(3,3)$ and $p\in\mathcal{P}$. Then:
\begin{itemize}
\item[(i)] the two following conditions are equivalent:
\begin{itemize}
	\item[(a)]the alternatives ranked first as well as those ranked third in $p$ are distinct;
	\item[(b)]$\Gamma_2(p)$ is a $3$-cycle.
\end{itemize}
\noindent Moreover, if one of the above conditions holds true, then the arc set of $\Gamma_3(p)$ is empty.
\item[(ii)]$\mu(p)=\mu(p^r).$
\end{itemize}

\end{lemma}

\begin{proof}
$(i)$ We start showing that $(a)$ implies $(b)$. Assume that $p_i(1)\neq p_j(1)$ and $p_i(3)\neq p_j(3)$ for all $i,j\in H=\{1,2,3\}$ with $i\neq j$. Without loss of generality  we can assume that $p_1(1)=1, p_2(1)=2, p_3(1)=3$. Thus $p_1(3)\in \{2,3\}$. If $p_1(3)=2$, then, since the alternatives ranked third are distinct, we necessarily have  $p_2(3)=3$ and $p_3(3)=1$. That implies that
\[
p=\left[
\begin{array}{cccccc}
1&2&3\\
3&1&2\\
2&3&1\\
\end{array}
\right]
\]
Similarly,  if $p_1(3)=3,$ we get \[
p=\left[
\begin{array}{cccccc}
1&2&3\\
2&3&1\\
3&1&2\\
\end{array}
\right]
\]
In both cases we have that $\Gamma_2(p)$ is a $3$-cycle and the arc set of $\Gamma_3(p)$ is empty.

We next show that $(b)$ implies $(a).$ Let $p\in \mathcal{P}$. If there exists $x\in N$  which is ranked first in $p$ by at least two individuals, then $x$ is a maximum for $\Gamma_2(p)$ and so it cannot be involved in a cycle of $\Gamma_2(p)$. 
 If there exists $x\in N$  which is ranked third in $p$ by at least two individuals, consider $p^r$. By what shown above,  $\Gamma_2(p^r)$ is acyclic, so that $\Gamma_2(p)$ is acyclic too.

 $(ii)$ By contradiction, assume that $\mu(p^r)\neq \mu(p),$ say $\mu(p^r)> \mu(p).$ Then $\mu(p)=2$ and $\mu(p^r)=3$. Thus $\Gamma_2(p^r)$ admits a cycle. By Lemma \ref{maximum}, we then get that  $\Gamma_2(p^r)$ is a $3$-cycle. Using $(i)$ we deduce that the alternatives ranked first as well as those ranked third in $p^r$ are distinct. But then, the same property holds for $p$, so that also $\Gamma_2(p)$ is a $3$-cycle. Thus $D_2(p)=\varnothing,$ against $\mu(p)=2$.
\end{proof}

\subsection{Proof of Proposition \ref{TM1}}\label{TM1-sec}

First of all, let us prove that if $(h,n)\in \mathbb{N}_\diamond^2\setminus T_1$, then there exist $p\in\mathcal{P}$ and $x\in N$
such that $D_{\mu(p)}(p)=D_{\mu(p^{r})}(p^{r})=\{x\}$.
We obtain the proof showing that the assumptions of Propositions \ref{blacktable-odd} or \ref{blacktable-even} hold true. First of all, note that  $(h,n)\in \mathbb{N}_\diamond^2\setminus T_1$ implies $h\geq 4$ and $n\geq 4.$
If $n=4$, then either $h$ is even with $h\geq 6$ and so satisfies $h\geq \frac{2(n-1)}{n-3}$,
 or $h$ is odd with $h\geq 9$ and so satisfies $h\geq \frac{3(n-1)}{n-3}$. If $n=5$, then the same argument applies.
If $n\geq 6,$ then we have $\frac{2(n-1)}{n-3}\leq 4\leq h$ for all $h$ even, as well as $\frac{3(n-1)}{n-3}\leq 5\leq h$ for all $h$ odd.

Assume now that  $(h,n)\in T_1$ and prove that it cannot be found $p\in\mathcal{P}$ and $x\in N$ such that $D_{\mu(p)}(p)=D_{\mu(p^{r})}(p^{r})=\{x\}$. Consider then $(h,n)\in T_1$ and assume, by contradiction, that there exist $p\in \mathcal{P}$ and $x\in N$  such that $ D_{\mu(p)}(p)=D_{\mu(p^{r})}(p^{r})=\{x\}.$ Since the Minimax {\sc scc} is neutral, we can assume that $x=n$ so that
\begin{equation}\label{D}
D_{\mu(p)}(p)=D_{\mu(p^{r})}(p^{r})=\{n\}.
\end{equation}
There are several cases to study.

If $n\in\{2,3\}$, then, by Corollary \ref{soglia-green}, we have that $\Gamma_{\mu(p)}(p)$ and  $\Gamma_{\mu(p^{r})}(p^{r})$ are both acyclic so that Lemma \ref{four} (iii) applies contradicting \eqref{D}.

If $h=2$, then $\mu(p)=\mu(p^r)=\mu_a=2$ and, by Corollary \ref{soglia-green}, we have that $\Gamma_{\mu(p)}(p)$ and  $\Gamma_{\mu(p^{r})}(p^{r})$ are both acyclic so that Lemma \ref{four} (iii) applies contradicting \eqref{D}.

If $h=3$, then $\mu_0=2$ and $\mu(p),\mu(p^{r})\in \{2,3\}$.
If $\mu(p)=2$, then, by Lemma \ref{minimalmu-ii}, we have that $\mu(p^{r})=3$ so that $D_{\mu(p^{r})}(p^{r})=\{n\}$ and $D_{\mu(p)}(p^{r})=\varnothing.$ Let $V=N\setminus \{n\}.$ Since $n$ is the only maximal element in $\Gamma_{\mu(p^{r})}(p^{r})$, for every $x\in V$, there exists $y\in N$ with $y>_{\mu(p^{r})}^{p^{r}}x$. Note that if $y$ were equal to $n$, then from $n>_{\mu(p^{r})}^{p^{r}}x$ we would get $x>_{\mu(p^{r})}^{p} n$ against the maximality of $n$ in  $\Gamma_{\mu(p)}(p)$. Thus,  there exists a cycle in $\Gamma_{\mu(p^{r})}(p^{r})$ involving some vertices of $V.$ That leads to a contradiction since, by Lemma \ref{ac}, $\Gamma_{\mu(p^{r})}(p^{r})$ is acyclic.
If $\mu(p^{r})=2$, then the previous argument applies to $p^{r}$. If $\mu(p)=\mu(p^{r})=3,$ then we reach a contradiction applying Lemma \ref{ac} and Corollary \ref{muacyclicnew}.

If $(h,n)=(4,4)$, then $\mu_0=\mu_a=3$ and  $\mu(p),\ \mu(p^{r})\in \{3,4\}$. Thus, by Corollary \ref{soglia-green}, $\Gamma_{\mu(p)}(p)$ and $\Gamma_{\mu(p^{r})}(p^{r})$ are both acyclic, so that Lemma \ref{four} (iii) applies contradicting \eqref{D}.

If $(h,n)=(5,4)$, then $\mu_0=3$, $\mu_{a}=\mu_{G}=4$, and $\mu(p),\ \mu(p^{r})\in\{3,4\}$.
If $\mu(p)=\mu(p^{r})=4$, then by Corollary \ref{soglia-green},
$\Gamma_{\mu(p)}(p)$ and $\Gamma_{\mu(p^{r})}(p^{r})$ are both acyclic and we contradict \eqref{D}, using Lemma \ref{four} (iii). If $\mu(p)=3=\mu_0$, then, by Lemma \ref{minimalmu-ii}, $\mu(p^{r})=4$. By Corollary \ref{soglia-green} we have that $\Gamma_{\mu(p^{r})}(p^{r})$ is acyclic and then, by Corollary \ref{cor-connected}, connected.
Assume there exists $x\in V=\{1,2,3\}$ such that $4>_{\mu(p^{r})}^{p^{r}} x$. Then $x>_{\mu(p^{r})}^p 4$, against $4\in D_{\mu(p)}(p)$. So, we have $4\not>_{\mu(p^{r})}^{p^{r}} x$, for all $x\in V$.
On the other hand, from $4\in D_{\mu(p^{r})}(p^{r})$, we deduce that $x\not>_{\mu(p^{r})}^{p^{r}} 4$. Thus,  $4$ is isolated in $\Gamma_{\mu(p^{r})}(p^{r})$, against the connection of $\Gamma_{\mu(p^{r})}(p^{r})$. If $\mu(p^r)=3=\mu_0$, then  the previous argument applies to $p^{r}$.

If $(h,n)=(7,4)$, then $\mu_{0}=4,$  $\mu_{a}=5$, $\mu_{G}=6$ and
 $\mu(p),\  \mu(p^{r})\in \{4,5,6\}.$
If $\mu(p),\ \mu(p^{r})\in \{5,6\}$, then by Corollary \ref{soglia-green},
$\Gamma_{\mu(p)}(p)$ and $\Gamma_{\mu(p^{r})}(p^{r})$ are both acyclic and we contradict \eqref{D}, using Lemma \ref{four} (iii).
If $\mu(p)=4$, then, by Lemma \ref{minimalmu-ii},  $\mu(p^{r})\in \{5,6\}.$ By Corollary \ref{soglia-green}, $\Gamma_{\mu(p^{r})}(p^{r})$ is acyclic and then, by Corollary \ref{cor-connected}, connected. Assume there exists $x\in V=\{1,2,3\}$ such that $4>_{\mu(p^{r})}^{p^{r}} x$. Then $x>_{\mu(p^{r})}^p 4$, against $4\in D_{\mu(p)}(p)$. So, we have
 $4\not >_{\mu(p^{r})}^{p^{r}} x$ for all $x\in V$.
 On the other hand, from $4\in D_{\mu(p^{r})}(p^{r})$ we deduce that $x\not>_{\mu(p^{r})}^{p^{r}} 4$ for all $x\in V.$  Thus, $4$ is isolated in $\Gamma_{\mu(p^{r})}(p^{r})$, against the connection of $\Gamma_{\mu(p^{r})}(p^{r})$. If $\mu(p^r)=4$, then  the previous argument applies to $p^{r}$.

If $(h,n)=(5,5)$, then $\mu_{0}=3,$ $\mu_a=4$, $\mu_{G}=5$ and $\mu(p),\  \mu(p^{r})\in \{3,4,5\}.$
If $\mu(p),\ \mu(p^{r})\in \{4,5\}$, then by Corollary \ref{soglia-green},
$\Gamma_{\mu(p)}(p)$ and $\Gamma_{\mu(p^{r})}(p^{r})$ are both acyclic and we contradict \eqref{D}, using Lemma \ref{four} (iii).
If $\mu(p)=3$, then, by Lemma \ref{minimalmu-ii}, $\mu(p^{r})\in\{4,5\}.$  By Corollary \ref{soglia-green}, $\Gamma_{\mu(p^{r})}(p^{r})$ is acyclic and then, by Corollary \ref{cor-connected}, connected.
Assume there exists $x\in V=\{1,2,3,4\}$ such that $5>_{\mu(p^{r})}^{p^{r}} x$. Then $x>_{\mu(p^{r})}^p 5$, against $5\in D_{\mu(p)}(p)$. On the other hand, from $5\in D_{\mu(p^{r})}(p^{r})$ we deduce that $x\not>_{\mu(p^{r})}^{p^{r}} 5$ for all $x\in V.$ Thus, $5$ is isolated in $\Gamma_{\mu(p^{r})}(p^{r})$, against the connection of $\Gamma_{\mu(p^{r})}(p^{r})$. If $\mu(p^r)=3$, then  the previous argument applies to $p^{r}$.

\subsection{Proof of Proposition \ref{TM2}}\label{TM2-sec}

First of all, let us prove that if $(h,n)\in \mathbb{N}_\diamond^2\setminus T_2$, then there exists $p\in\mathcal{P}$ such that $D_{\mu(p)}(p)=\{1\}\subseteq D_{\mu(p^{r})}(p^{r})$.
If $(h,n)\in\mathbb{N}_\diamond^2\setminus T_1$, then we can apply Proposition \ref{TM1}. Assume then that  $(h,n)\in T_1\setminus T_2$ and note that
\[
T_1\setminus T_2=\{(h,n)\in\mathbb{N}_\diamond^2: h=3,\,n\ge 4\}\cup\{(5,4), (5,5),(7,4)\}.
\]
If $(h,n)\in\mathbb{N}_\diamond^2$ is such that $h=3$ and $n\ge 4$, then consider $p\in\mathcal{P}$
defined by
\[
p_1=[1,(5),\ldots, (n),2,3,4]^T, \quad p_2=[1,(5),\ldots, (n),3,4,2]^T,\quad p_3=[4,2,3,(n),\ldots, (5),1]^T
\]
Thus, $\mu(p)=2$ and $D_{\mu(p)}(p)=\{1\}$, while $\mu(p^r)=3$ and $D_{\mu(p^r)}(p^r)=N.$

If $(h,n)=(5,4)$, then consider $p\in\mathcal{P}$ defined by
\[
\left[
\begin{array}{cccccccc}
1&1&1&2&3\\
2&3&4&3&4\\
3&4&2&4&2\\
4&2&3&1&1
\end{array}
\right]
\]
Thus, $\mu(p)=3$ and $D_{\mu(p)}(p)=\{1\}$, while $\mu(p^r)=4$ and $D_{\mu(p^r)}(p^r)=\{1,2,4\}.$

If $(h,n)=(5,5)$, then consider $p\in\mathcal{P}$ defined by
\[
\left[
\begin{array}{cccccccc}
1&1&1&5&2\\
2&3&4&2&3\\
3&4&5&3&4\\
4&5&2&4&5\\
5&2&3&1&1
\end{array}
\right]
\]
Thus, $\mu(p)=3$ and $D_{\mu(p)}(p)=\{1\}$, while $\mu(p^r)=4$ and $D_{\mu(p^r)}(p^r)=\{1,5\}.$

If $(h,n)=(7,4)$, then consider $p\in\mathcal{P}$ defined by
\[
\left[
\begin{array}{cccccccc}
1&1&1&1&3&4&2\\
2&3&4&2&4&2&3\\
3&4&2&3&2&3&4\\
4&2&3&4&1&1&1\\
\end{array}
\right]
\]
Thus,  $\mu(p)=4$ and $D_{\mu(p)}(p)=\{1\}$, while $\mu(p^r)=5$ and $D_{\mu(p^r)}(p^r)=\{1,2,4\}.$

Assume now that $(h,n)\in T_2$ and prove that it cannot be found $p\in\mathcal{P}$ and $x\in N$ such that $D_{\mu(p)}(p)=\{x\}\subseteq D_{\mu(p^{r})}(p^{r})$. By Lemmata  \ref{four}(i) and \ref{intersection3}(ii), it is enough to show that $\Gamma_{\mu(p)}(p)$ and $\Gamma_{\mu(p^r)}(p^r)$ are both acyclic. This comes applying Corollary \ref{soglia-green} in all the possible cases. The application is obvious when $n\in\{2,3\}$; if $h=2$ note that $\mu(p)=\mu(p^r)=\mu_a=2$; if $(h,n)=(4,4)$ note that $\mu_0=\mu_a=3$ and thus $\mu(p),  \mu(p^r)\geq 3.$

\subsection{Proof of Proposition \ref{TM3}}\label{TM3-sec}

First of all, let us prove that if $(h,n)\in \mathbb{N}_\diamond^2\setminus T_3$, then there exists $p\in\mathcal{P}$ such that $D_{\mu(p)}(p)\neq N$ and $D_{\mu(p)}(p)\cap D_{\mu(p^{r})}(p^{r})\neq \varnothing$.
If $(h,n)\in \mathbb{N}_\diamond^2\setminus T_2$ then we can apply Proposition \ref{TM2}. Assume then that  $(h,n)\in T_2\setminus T_3$ and note that
\[
T_2\setminus T_3=\{(h,n)\in\mathbb{N}_\diamond^2: h=2,\,n\ge 3\}\cup \{(h,n)\in\mathbb{N}_\diamond^2: h\neq 3,\, n=3\}\cup\{(4,4)\}.
\]
If $(h,n)\in \mathbb{N}_\diamond^2$ is such that $h=2$ and $n\ge 3$, then consider
$p\in\mathcal{P}$ defined by
\[
p_1=[1,2,3,\ldots,n-1,n]^T,\quad p_2=[n,1,2,\ldots,n-1]^T.
\]
Thus, $\mu(p)=\mu(p^r)=2$ and, since $n\ge 3$, we have $D_{\mu(p)}(p)=\{1,n\}\neq N$. Moreover
 $D_{\mu(p^r)}(p^r)=\{n-1,n\}$, so that $D_{\mu(p)}(p)\cap D_{\mu(p^{r})}(p^{r})=\{n\}.$

If $(h,n)\in \mathbb{N}_\diamond^2$ is such that $h\neq 3$ and $n=3$, then consider the partition of $\mathbb{N}_\diamond\setminus\{3\}$ given by $H_1=\{h=2+3k:k\geq 0\}$, $H_2=\{h=1+3k:k\geq 1\}$, $H_3=\{h=3+3k:k\geq 1\}$.
If $h\in H_1,$ then consider any $p\in\mathcal{P}$ such that
\[
\begin{array}{l}
|\{i\in H:p_i=[1,2,3]^T\}|=1+k,\quad |\{i\in H:p_i=[3,1,2]^T\}|=1+k,\\
\vspace{-2mm}\\
  |\{i\in H:p_i=[2,3,1]^T\}|=k.
\end{array}
\]
If $h\in H_2,$ consider any $p\in\mathcal{P}$ such that
\[
\begin{array}{l}
|\{i\in H:p_i=[1,2,3]^T\}|=k,\quad |\{i\in H:p_i=[2,3,1]^T\}|=k, \\
\vspace{-2mm}\\
|\{i\in H:p_i=[3,1,2]^T\}|=k,\quad  |\{i\in H:p_i=[1,3,2]^T\}|=1.
\end{array}
\]
If $h\in H_3,$ consider any $p\in\mathcal{P}$ such that
\[
\begin{array}{l}
|\{i\in H:p_i=[1,2,3]^T\}|=k,\quad  |\{i\in H:p_i=[3,1,2]^T\}|=k,\\
\vspace{-2mm}\\
|\{i\in H:p_i=[2,3,1]^T\}|=k+1,\quad  |\{i\in H:p_i=[1,3,2]^T\}|=2.
\end{array}
\]
In all the above situations, it is easily checked that $D_{\mu(p)}(p)=\{1,3\}\neq N$ and  $D_{\mu(p^r)}(p^r)=\{2,3\}$ so that $D_{\mu(p)}(p)\cap D_{\mu(p^r)}(p^r)=\{3\}\neq \varnothing$.

If $(h,n)=(4,4)$, then consider $p\in\mathcal{P}$ defined by
\[
\left[
\begin{array}{cccccccc}
1&1&4&4\\
2&2&2&2\\
3&3&3&3\\
4&4&1&1\\
\end{array}
\right]
\]
Thus $\mu(p)=\mu(p^r)=3$, $D_{\mu(p)}(p)=\{1,2,4\}\neq N$ and $D_{\mu(p^r)}(p^r)=\{1,3,4\}.$

Assume now that $(h,n)\in T_3$ and prove that it cannot be found $p\in\mathcal{P}$ such that $D_{\mu(p)}(p)\neq N$ and $D_{\mu(p)}(p)\cap D_{\mu(p^{r})}(p^{r})=\varnothing$.

If $n=2$, then the condition $D_{\mu(p)}(p)\neq N$ is equivalent to $|D_{\mu(p)}(p)|=1$. Since $(h,n)\in T_2$ Proposition \ref{TM2} applies.

Finally let $(h,n)=(3,3)$. We show that, for every $p\in \mathcal{P}$, we have $M(p)\cap M(p^r)=\varnothing$  or $M(p)=N$. Fix $p\in\mathcal{P}$ and  note that $\mu_0=2$.
Assume first that there exists $x\in N$ such that  $\{i\in H:p_i(1)=x\}$  has at least two elements. By Lemma \ref{lemma33} (i) and Lemma \ref{maximum}, $\Gamma_2(p)$  is acyclic. Thus $\mu(p)=2$ and $M(p)=D_2(p)=\{1\}.$ By Lemma \ref{lemma33} (ii) we also have $\mu(p^r)=2$. Since in $p^r$ the alternative $1$ is beaten by the alternative $2$ at least two times, we have that $1\notin D_2(p^r)=M(p^r)$ and so $M(p)\cap M(p^r)=\varnothing$.
If there exists $x\in N$ such that  $\{i\in H:p_i(3)=x\}$ has at least two elements we apply the argument above to $p^r$, obtaining again $M(p)\cap M(p^r)=\varnothing$.
We are then left with assuming that the alternatives ranked first as well as those ranked third are distinct in $p$. In this case, by Lemma \ref{lemma33} (i), $\Gamma_2(p)$ is a $3$-cycle and $\mu(p)=3$. Moreover, the arc set of $\Gamma_3(p)$ is empty so that $M(p)=D_3(p)=N$. \vspace{7mm}

\noindent {\Large{\bf References}}
\vspace{2mm}

\noindent Bubboloni, D., Gori, M., 2015. Symmetric majority rules. Mathematical Social Sciences 76, 73-86.
\vspace{2mm}

\noindent Diestel, R., 2010. {\it  Graph Theory}, 4th edition. Graduate Texts in Mathematics  173, Springer-Verlag, Heidelberg.
\vspace{2mm}

\noindent Fishburn, P.C., 1977. Condorcet social choice functions. SIAM Journal on Applied Mathematics 33, 469-489.
\vspace{2mm}

\noindent Greenberg, J., 1979. Consistent majority rules over compact sets of alternatives. Econometrica 47, 627-636.
\vspace{2mm}

\noindent Laslier, J.-F., 1997. {\it Tournament solutions and majority voting}. Studies in Economic Theory, Volume 7. Springer.
\vspace{2mm}

\noindent  Llamazares, B., Pe$\mathrm{\tilde{n}}$a, T., 2015. Scoring rules and social choice properties: some characterizations. Theory and Decision 78, 429-450.
\vspace{2mm}

\noindent  Miller, N.R., 1977. Graph-theoretical approaches to the theory of voting. American Journal of Political Science 21, 769-803.
\vspace{2mm}

\noindent Saari, D.G., 1994.  {\it Geometry of voting}.  Studies in Economic Theory, Volume 3. Springer.
\vspace{2mm}

\noindent Saari, D.G., Barney, S., 2003. Consequences of reversing preferences. The Mathematical Intelligencer 25, 17-31.

\end{document}